\documentclass[12pt,reqno]{amsart}

\usepackage{multirow}
\usepackage{hhline}

\usepackage{mathtools}
\usepackage{times}
\usepackage[T1]{fontenc}
\usepackage{mathrsfs}
\usepackage{latexsym}
\usepackage[dvips]{graphics}
\usepackage[titletoc, title]{appendix}
\usepackage{epsfig}
\usepackage{amssymb}
\usepackage{amsmath,amsfonts,amsthm,amssymb,amscd}
\usepackage{color}
\usepackage{hyperref}
\usepackage{amsmath}
\usepackage{float}

\usepackage{color}
\usepackage{breakurl}

\usepackage{comment}
\newcommand{\bburl}[1]{\textcolor{blue}{\url{#1}}}

\newcommand{\be}{\begin{equation}}
\newcommand{\ee}{\end{equation}}
\newcommand{\seqnum}[1]{\href{https://oeis.org/#1}{\rm \underline{#1}}}

\newtheorem{thm}{Theorem}[section]

\newtheorem{lem}[thm]{Lemma}

\newtheorem{exa}[thm]{Example}

\newtheorem{rek}[thm]{Remark}

\numberwithin{equation}{section}

\newcommand{\Mod}[1]{\ \mathrm{mod}\ #1}

\begin{document}

{\color{blue} }

\title[Schreier Sets and Tur\'{a}n Graphs]{On a Relation between Schreier-type sets and a Modification of Tur\'{a}n Graphs}

\author{H\`ung Vi\d{\^e}t Chu}

\email{\textcolor{blue}{\href{mailto:hungchu2@illinois.edu}{hungchu2@illinois.edu}}}
\address{Department of Mathematics, University of Illinois at Urbana-Champaign, Urbana, IL 61820, USA}



\maketitle

\begin{abstract} 
Recently, a relation between Schreier-type sets and Tur\'{a}n graphs was discovered. In this note, we give a combinatorial proof and obtain a generalization of the relation. Specifically, for $p, q\ge 1$, let $$\mathcal{A}_q := \{F\subset\mathbb{N}: |F| = 1 \mbox{ or }F\mbox{ is an arithmetic progression with difference } q\}$$
and 
$$Sr(n, p, q)\ :=\ \#\{F\subset \{1, \ldots, n\}\,:\, p\min F\ge |F|\mbox{ and }F\in \mathcal{A}_q\}.$$
We show that 
$$Sr(n, p, q) \ =\ T(n+1, pq+1, q),$$
where $T(\cdot, \cdot, \cdot)$ is the number of edges of a modification of Tur\'{a}n graphs. We also prove that $Sr(n,p,q)$ is the partial sum of certain sequences. 
\end{abstract}

\section{Introduction}

A set $F\subset \mathbb{N}$ is said to be $\textit{Schreier}$ if $\min F\ge |F|$. Schreier sets have appeared  frequently in Banach space theory. For example, Schreier sets were used in the construction by Schreier to answer a question of Banach and Saks \cite{S} and were the central concept in a celebrated theorem by Odell \cite{Od}. Besides, Schreier sets are of independent interests. 
A. Bird \cite{B} showed a surprising connection between these sets and the Fibonacci sequence. Since then, there has been research on generalizing Bird's result to connect Schreier-type sets with other sequences (see \cite{C1, C2, C3, M}.) For instance, \cite[Theorem 1.1]{C2} investigated the recurrence relation for the count of sets $F\subset\mathbb{N}$ satisfying $\min F\ge p|F|$, where $p\ge 1$. Recently, Beanland et al. \cite{BCF} proved a recurrence for the more general form $q\min F\ge p|F|$, where $p, q\ge 1$. In the same note, the authors showed a connection between Schreier-type sets and Tur\'{a}n graphs, which we now describe. 

For $n, p\ge 1$, let $[n] = \{1, 2, \ldots, n\}$ and 
$$Sr(n,p)\ :=\ \#\{F\subset [n]\,:\, p\min F\ge |F|\mbox{ and }F\mbox{ is an interval}\}.$$
A Tur\'{a}n graph \cite{T}, denoted by $T(n, p)$, is the $n$-vertex complete $p$-partite graph whose parts differ in size by at most $1$ vertex. In other words, to form $T(n, p)$, we divide the $n$ vertices into $p$ parts as equally as possible then connect all pairs of vertices that are not in the same part. We also use $T(n,p)$ to indicate the number of edges of the Tur\'{a}n graph $T(n,p)$. For each $p\ge 2$, the sequence $(T(n,p))_{n=1}^\infty$ is available in OEIS \cite{Sln} (for example, see
\seqnum{A002620},
\seqnum{A000212},
\seqnum{A033436}, and
\seqnum{A033437}.)

By \cite[Theorem 1.2]{BCF}, we have
\begin{equation}\label{e1}Sr(n,p)\ =\ T(n+1, p+1).\end{equation}
In the proof of \cite[Theorem 1.2]{BCF}, the authors explicitly found the formulas for $Sr(n,p)$ and $T(n+1, p+1)$ then by algebraic manipulations, showed that the two formulas give the same number for all $n, p\ge 1$. 
The first goal of this note is to give a combinatorial proof of \eqref{e1}, which is concise and sheds a better light on why the equality holds. 

Our second goal is to generalize \eqref{e1}. To do so, we need to introduce some new notion. 
\subsection{Three types of graphs}
We introduce several types of graphs that are recursively defined. 

\subsubsection{An $M(n,p)$-graph}
In a Tur\'{a}n graph, each vertex in a particular part shares an edge only with vertices in other parts. We shall present a slight modification that allows a vertex to share an edge with vertices in the same part. Let $M(n,p)$ denote the following modification of $T(n,p)$. Fix $p\ge 1$. Let $M(1, p)$ be the graph with $p$ parts, exactly one of which contains a vertex. Suppose that $M(n,p)$ has been defined for some $n\ge 1$. We construct $M(n+1, p)$ by adding a vertex $v$ to $M(n,p)$ under a certain rule. Write $n = p\ell+k$, where $\ell\ge 0$ and $1\le k\le p$. 
\begin{itemize}
    \item Case 1: If $k = p$, then add $v$ to one of the $p$ parts. From the remaining $p-1$ parts, choose an arbitrary part $P$. Form edges between $v$ and all vertices in the same part as $v$, and form edges between $v$ and vertices in all other parts except $P$. 
    \item Case 2: If $k < p$, then add $v$ to a part $P'$ with exactly $\ell$ vertices. Form an edge between $v$ and every vertex in every part other than $P'$. 
\end{itemize}
With an abuse of notation, let $M(n,p)$ also denote the number of edges of $M(n,p)$. Clearly, $M(n,p) = T(n,p)$ for all $n, p\ge 1$.

\subsubsection{An $M(n,p,q)$-graph}

For fixed $p, q\ge 1$, we now define a graph $M(n,p,q)$, where the number of edges of $M(n,p,1)$ and $M(n,p)$ are equal. Let $M(1, p, q)$ be the graph with $p$ parts, exactly one of which contains a vertex. Suppose that $M(n,p, q)$ has been defined for some $n\ge 1$. We construct $M(n+1, p, q)$ by adding a vertex $v$ to $M(n,p, q)$ under the following rule. Write $n = p\ell + k$, where $\ell\ge 0$ and $1\le k\le p$. There are $k$ parts with $\ell+1$ vertices and $p-k$ parts with $\ell$ vertices. 
\begin{itemize}
    \item Case 1: If $p = k$, then form $M(n+1, p, q)$ in the same way as we form $M(n+1,p)$ above.
    \item Case 2: If $p-q< k < p$, then add $v$ to a part with exactly $\ell$ vertices. Let $P$ be a part with $\ell+1$ vertices. Form edges between $v$ and all vertices in the same part as $v$, and form edges between $v$ and vertices in all other parts except $P$.
    \item Case 3: If $k \le p-q$, then add $v$ to a part $P'$ with exactly $\ell$ vertices. Form an edge between $v$ and every vertex in every part other than $P'$.
\end{itemize}
By definition, $M(n, p, 1) = M(n,p)$ for all $n,p\ge 1$.

\subsubsection{A $T(n, p, q)$-graph}
A $T(n,p,q)$-graph is formed in almost the same way as an $M(n,p,q)$-graph. Fix $p, q\ge 1$. Let us describe how to form $T(n,p,q)$. First, $T(1, p, q) = M(1, p, q)$. Supposing that we already have $T(n, p, q)$ for some $n\ge 1$, we construct $T(n+1, p, q)$.
In each case of the recursive step to form $M(n+1, p, q)$ described above, we connect $v$ to certain vertices. Call the set of these vertices $V$. We assign numbers from $1$ to $q$ in that order to vertices in $V$. (If $|V| > q$, we repeat the numbering from $1$ to $q$.) In forming $T(n+1, p, q)$ from $T(n, p, q)$, we form an edge between $v$ and vertices that are numbered $1$. See the Appendix for $(T(n, 5, 2))_{n=2}^{7}$.

\subsection{Schreier sets with constant gaps}

Recall that an arithmetic progression of positive integers is a sequence that can be written as $$a, a+d, a+2d, \ldots, a+kd, \ldots,$$
for some $a, d\in \mathbb{N}$. Here $d$ is called the difference of the arithmetic progression. An interval is an arithmetic progression with difference $1$. For $q\ge 1$, let $$\mathcal{A}_q := \{F\subset\mathbb{N}: |F| = 1 \mbox{ or }F\mbox{ is an arithmetic progression with difference } q\}.$$
For $n, p, q\ge 1$, define
$$Sr(n, p, q)\ :=\ \#\{F\subset [n]\,:\, p\min F\ge |F|\mbox{ and }F\in \mathcal{A}_q\}.$$

We are ready to state the second result. 

\begin{thm}\label{m1}
For all $p, q, n\ge 1$, we have
\begin{equation}\label{e10}Sr(n, p, q) \ =\ T(n+1, pq+1, q).\end{equation}
\end{thm}

\begin{rek}\normalfont
Plugging in $q = 1$ into \eqref{e10}, we have \eqref{e1}. Indeed, we have $Sr(n, p, 1) = Sr(n, p)$ and by definitions, 
$$T(n+1, p+1, 1) \ =\ M(n+1, p+1, 1)\ =\ M(n+1, p+1) \ =\ T(n+1, p+1).$$
\end{rek}

\begin{exa}\normalfont
Choose $p = q = 2$. Let us consider the first few terms of two sequences $(Sr(n, 2, 2))_{n=1}^\infty$ and $(T(n+1, 5, 2))_{n=1}^{\infty}$ and confirm that these values are equal. A simple program gives $(Sr(n, 2, 2))$:
$$1, 2, 4, 6, 8, 11, 14, 18, 22, 26, 31, 36, 42, 48, 54, 61, 68, 76, 84, \ldots.$$
We include graphs $T(n,5,2)$ for small values of $n$ in the Appendix. 
\end{exa}

Our last result involves writing $Sr(n,p,q)$ as partial sums of certain sequences. In particular,
\begin{thm}\label{m2}
Fix $p, q\ge 1$. We have
$$Sr(1, p, q) \ =\ 1\mbox{ and }Sr(n+1, p, q) - Sr(n, p, q) \ =\ \left\lfloor\frac{p(n+q+1)}{pq+1}\right\rfloor, \forall n\ge 1.$$
As a result,
$$Sr(N, p, q)\ =\ 1 + \sum_{n=1}^{N-1}\left\lfloor\frac{p(n+q+1)}{pq+1}\right\rfloor.$$
\end{thm}

\section{A combinatorial view of \eqref{e1}}
Let $p\ge 1$ be fixed. Our goal is to show that $Sr(n, p) = T(n+1, p+1)$ for all $n\ge 1$. The proof idea is simple: as we know $Sr(1, p) = T(2, p+1)$ from the definitions, we need only to verify that 
\begin{equation}\label{e2}Sr(n+1, p) - Sr(n, p) \ =\ T(n+2, p+1) - T(n+1, p+1),\forall n\ge 1.\end{equation}
Fix $n\ge 1$. Write $n = (p+1)\ell + k$ for some $\ell\ge 0$ and $0\le k\le p$.

By definition, the left side of \eqref{e2} counts the number of intervals $F\subset [n+1]$ such that 
$p\min F\ge |F|$ and $n+1\in F$. Let $\mathcal{A}$ be the set of all these intervals and $M$ be the largest interval in $\mathcal{A}$. Then 
$$p\min M\ \ge\ |M| \ =\  n+2-\min M.$$
Hence, 
$$\min M \ =\ \left\lceil \frac{n+2}{p+1}\right\rceil\mbox{ and }M \ =\ \left\{\left\lceil \frac{n+2}{p+1}\right\rceil, \ldots, n+1\right\}.$$
Since each $F\in \mathcal{A}$ is obtained from $M$ by discarding the smallest numbers in $M$, it follows that
\begin{align}\label{e11}Sr(n+1, p) - Sr(n, p)\ =\ n+2-\left\lceil \frac{n+2}{p+1}\right\rceil&\ =\ n-\ell + 2-\left\lceil\frac{k+2}{p+1}\right\rceil\nonumber\\
&\ =\ \begin{cases}n-\ell &\mbox{ if }k = p\\
n-\ell+1 &\mbox{ if } k < p.\end{cases}
\end{align}

We now consider the right side of \eqref{e2}. For the graph $T(n+1, p+1)$, there are $k+1$ parts that have $\ell+1$ vertices and $(p-k)$ parts that have $\ell$ vertices. If $k = p$, then adding one more vertex to any part increases the number of edges by $$(\ell+1)k \ =\ (\ell+1)p \ =\ p\ell + k\ =\  n-\ell.$$
If $k < p$, then adding one more vertex to a part with $\ell$ vertices increases the number of edges by 
$$\ell(p-k-1) + (\ell+1)(k+1)\ =\ n-\ell + 1.$$
Therefore, we have $Sr(n+1, p) - Sr(n, p) = T(n+2, p+1) - T(n+1, p+1)$, as desired. 

\section{Proof of Theorem \ref{m1}}
First, we need an easy lemma.

\begin{lem}\label{l1}
Given a set $V$ of $N$ vertices and a vertex $v\notin V$, if we number vertices in $V$ from $1$ to $q$ in that order and form an edge between $v$ and each vertex that is numbered $1$, then the number of edges is 
$$\left\lfloor \frac{N-1}{q}\right\rfloor + 1.$$
\end{lem}
\begin{proof}
We can rephrase the problem as follows: given $N$ integers, find the maximum possible number of integers that are $1\Mod q$. The count is $\left\lfloor \frac{N-1}{q}\right\rfloor + 1$.
\end{proof}

\begin{proof}[Proof of Theorem \ref{m1}]
Fix $p, q, n\ge 1$. Write $n = (pq+1)\ell + k$ for some $\ell\ge 0$ and $0\le k\le pq$. 
Clearly, $Sr(1, p, q) \ =\ T(2, pq+1, q) = 1$. We need only to show that 
$$Sr(n+1, p, q) - Sr(n, p, q) \ =\ T(n+2, pq+1, q) - T(n+1, pq+1, q).$$

We have $Sr(n+1, p, q) - Sr(n, p, q)$ counts the number of sets $F\subset [n+1]$ with $\max F = n+1$, $p\min F\ge |F|$, and $F\in \mathcal{A}_q$. Let $\mathcal{A}$ be the set of all such $F$ and $M$ be the largest set in $\mathcal{A}$. Then 
$$p\min M\ \ge\ |M| \ =\  \frac{n+1-\min M}{q}+1$$
and so, $$\min M\ \ge\ \left\lceil\frac{n+1+q}{pq+1}\right\rceil.$$ 
If we write $M = \{n+1-(\ell-1) q,\ldots, n+1-q, n+1\}$, then 
$$n+1-(\ell-1)q \ \ge\ \left\lceil\frac{n+1+q}{pq+1}\right\rceil,$$
which gives 
$$\ell \ \le\ \left\lfloor \frac{1}{q}\left(n+1-\left\lceil\frac{n+1+q}{pq+1}\right\rceil\right)\right\rfloor+1.$$
Because $M$ is the largest, 
$$|M| \ =\ \ell \ =\ \left\lfloor \frac{1}{q}\left(n+1-\left\lceil\frac{n+1+q}{pq+1}\right\rceil\right)\right\rfloor+1.$$
Since each $F\in \mathcal{A}$ is obtained from $M$ by discarding the smallest numbers in $M$, it follows that
\begin{align}\label{e14}
    Sr(n+1, p, q) - Sr(n, p, q) &\ =\ \left\lfloor \frac{1}{q}\left(n+1-\left\lceil\frac{n+1+q}{pq+1}\right\rceil\right)\right\rfloor+1\nonumber\\
    &\ =\ \left\lfloor \frac{1}{q} \left(n-\ell+1-\left\lceil\frac{k+1+q}{pq+1}\right\rceil\right)\right\rfloor + 1\nonumber \\
    &\ =\ \begin{cases} \left\lfloor\frac{n-\ell -1 }{q}\right\rfloor + 1 &\mbox{ if } (p-1)q < k \le pq\\ \left\lfloor\frac{n-\ell}{q}\right\rfloor + 1&\mbox{ if }k \le  (p-1)q.\end{cases}
\end{align}

We now evaluate $T(n+2, pq+1, q) - T(n+1, pq+1, q)$. Again, $n+1 = (pq+1)\ell + (k+1)$, where $1\le k+1\le pq + 1$.
For the graph $T(n+1, pq+1, q)$, there are $k+1$ parts that have $\ell+1$ vertices and $(pq-k)$ parts that have $\ell$ vertices.

\begin{itemize}
    \item Case 1: If $k = pq$, then the new vertex in forming $T(n+2, pq+1, q)$ can be added to any of the $pq+1$ parts. By the construction of $T(n+2, pq+1, q)$ from $T(n+1, pq+1, q)$ and Lemma \ref{l1}, the number of new edges is 
    $$\left\lfloor\frac{pq(\ell+1)-1}{q}\right\rfloor + 1\ =\ \left\lfloor\frac{n-\ell-1}{q}\right\rfloor + 1.$$
    \item Case 2: If $(p-1)q < k < pq$, then the new vertex in forming $T(n+2, pq+1, q)$ must be added to one of the $(pq-k)$ parts that have $\ell$ vertices. 
    By the construction of $T(n+2, pq+1, q)$ from $T(n+1, pq+1, q)$ and Lemma \ref{l1}, the number of new edges is
    $$\left\lfloor\frac{\ell(pq-k)+(\ell+1)k-1}{q}\right\rfloor + 1\ =\ \left\lfloor\frac{n-\ell-1}{q}\right\rfloor + 1.$$
    \item Case 3: If $k\le (p-1)q$, then the new vertex in forming $T(n+2, pq+1, q)$ must be added to one of the $(pq-k)$ parts that have $\ell$ vertices. By the construction of $T(n+2, pq+1, q)$ from $T(n+1, pq+1, q)$ and Lemma \ref{l1}, the number of new edges is
    $$\left\lfloor\frac{\ell(pq-k-1)+(\ell+1)(k+1)-1}{q}\right\rfloor + 1\ =\ \left\lfloor\frac{n-\ell}{q}\right\rfloor + 1.$$
\end{itemize}
Therefore, we have $Sr(n+1, p, q) - Sr(n, p, q) = T(n+2, pq+1, q) - T(n+1, pq+1, q)$, as desired. 
\end{proof}

\section{Proof of Theorem \ref{m2}}
Fix $p, q\ge 1$.
We need the following lemma. 
\begin{lem}\label{l2}
For $k\ge 1$ and $k\le (p-1)q$, it holds that
$$\left\lfloor \frac{k}{q}\right\rfloor\ =\ \left\lfloor\frac{pk + p - 1}{pq+1}\right\rfloor.$$
\end{lem}
\begin{proof}
Write $k = qs + t$, where $s\ge 0$ and $0\le t < q$. Since $k \le (p-1)q$, we know that $s\le p-1$. Clearly, $\left\lfloor \frac{k}{q}\right\rfloor = s$. Also,
$$\left\lfloor\frac{pk + p - 1}{pq+1}\right\rfloor\ =\ \left\lfloor\frac{(1+pq)s + (p-1) -s + pt }{pq+1}\right\rfloor\ =\ s + \left\lfloor\frac{(p-1) -s + pt }{pq+1}\right\rfloor\ =\ s,$$
where the last equality is due to the fact that
$(p-1)-s + pt < pq + 1$, which is equivalent to 
$-s < p(q-t-1) + 2$.
\end{proof}

We want to show that
$$Sr(n+1, p, q) - Sr(n, p, q) \ =\ \left\lfloor \frac{p(n+q+1)}{pq+1}\right\rfloor, \forall n\ge 1.$$
By \eqref{e14}, we have
\begin{align*}
    Sr(n+1, p, q)- Sr(n, p, q) \ =\ \begin{cases} \left\lfloor\frac{n-\ell -1 }{q}\right\rfloor + 1 &\mbox{ if } (p-1)q < k \le pq\\ \left\lfloor\frac{n-\ell}{q}\right\rfloor + 1&\mbox{ if }k \le  (p-1)q,\end{cases}
\end{align*}
where $n = (pq+1)\ell + k$ for some $\ell \ge 0$ and $0\le k\le pq$. Substituting $n = (pq+1)\ell+k$, we obtain
\begin{equation}\label{e20}
Sr(n+1, p, q)- Sr(n, p, q)\ =\ \begin{cases} p\ell + \left\lfloor\frac{k -1 }{q}\right\rfloor + 1 &\mbox{ if } (p-1)q < k \le pq\\ p\ell + \left\lfloor\frac{k}{q}\right\rfloor + 1&\mbox{ if }k \le  (p-1)q.\end{cases}
\end{equation}
On the other hand,
\begin{equation}\label{e21}
\left\lfloor \frac{p(n+q+1)}{pq+1}\right\rfloor\ =\ \left\lfloor \frac{p(((pq+1)\ell+k)+q+1)}{pq+1}\right\rfloor\ =\ p\ell + 1 + \left\lfloor \frac{pk + p - 1}{pq+1}\right\rfloor.
\end{equation}
By \eqref{e20} and \eqref{e21}, we need only to confirm that 
$$\left\lfloor \frac{pk + p - 1}{pq+1}\right\rfloor\ =\ \begin{cases} \left\lfloor\frac{k -1 }{q}\right\rfloor &\mbox{ if } (p-1)q < k \le pq\\  \left\lfloor\frac{k}{q}\right\rfloor &\mbox{ if }k \le  (p-1)q.\end{cases}.$$
If $(p-1)q < k\le pq$, it is easy to see that $\left\lfloor \frac{pk + p - 1}{pq+1}\right\rfloor = \left\lfloor\frac{k -1 }{q}\right\rfloor = p-1$. If $k\le (p-1)q$, then we Lemma \ref{l2} and we are done. 

\section{Appendix}
The following are $T(n, 5, 2)$ for $2\le n\le 7$.

\begin{figure}[H]
\centering
\includegraphics[scale=1.3]{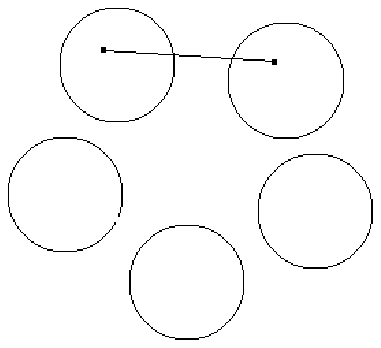}
\end{figure}

\begin{figure}[H]
\centering
\includegraphics[scale=1.3]{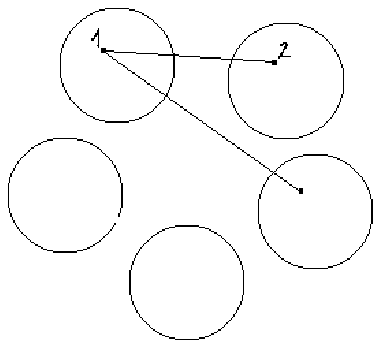}
\end{figure}

\begin{figure}[H]
\centering
\includegraphics[scale=1.3]{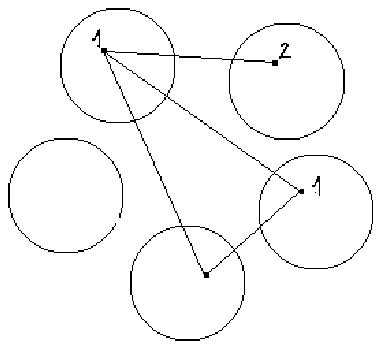}
\end{figure}

\begin{figure}[H]
\centering
\includegraphics[scale=1.1]{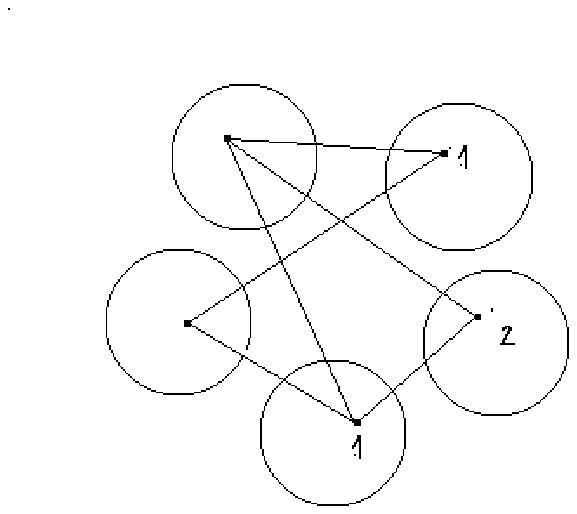}
\end{figure}

\begin{figure}[H]
\begin{center}
\includegraphics[scale=0.6]{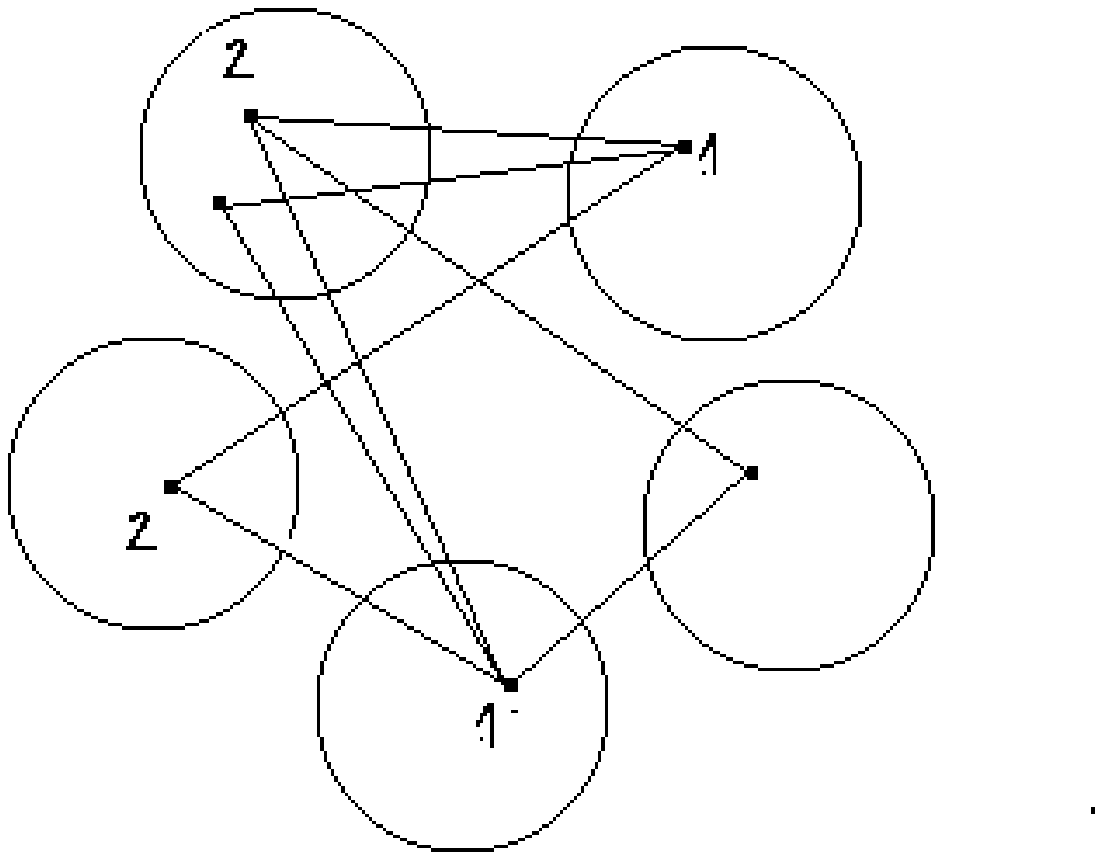}
\end{center}
\end{figure}

\begin{figure}[H]
\centering
\includegraphics[scale=1.1]{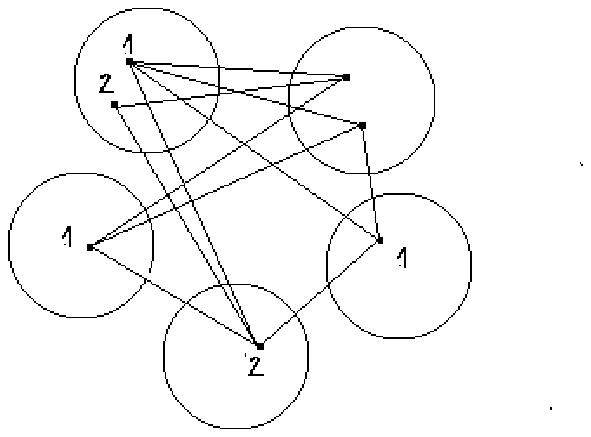}
\end{figure}


\ \\
\end{document}